\documentclass[11pt]{amsart}

\usepackage{amssymb}
\usepackage{amsmath}
\usepackage[all,cmtip]{xy}
\usepackage{amsfonts}
\usepackage{mathrsfs}
\usepackage{latexsym}
\usepackage{graphicx}
\usepackage{amscd,amssymb,amsmath,amsbsy,amsthm}
\usepackage[colorlinks,plainpages,backref,urlcolor=blue]{hyperref}
\allowdisplaybreaks

\usepackage{tikz}
\usetikzlibrary{arrows,calc}
\tikzset{
>=stealth',
help lines/.style={dashed, thick},
axis/.style={<->},
important line/.style={thick},
connection/.style={thick, dotted},
}

\newcommand{\nc}{\newcommand}
\nc{\rnc}{\renewcommand}
\nc{\bb}[1]{{\mathbb #1}}
\nc{\bbA}{\bb{A}}\nc{\bbB}{\bb{B}}\nc{\bbC}{\bb{C}}\nc{\bbD}{\bb{D}}
\nc{\bbE}{\bb{E}}\nc{\bbF}{\bb{F}}\nc{\bbG}{\bb{G}}\nc{\bbH}{\bb{H}}
\nc{\bbI}{\bb{I}}\nc{\bbJ}{\bb{J}}\nc{\bbK}{\bb{K}}\nc{\bbL}{\bb{L}}
\nc{\bbM}{\bb{M}}\nc{\bbN}{\bb{N}}\nc{\bbO}{\bb{O}}\nc{\bbP}{\bb{P}}
\nc{\bbQ}{\bb{Q}}\nc{\bbR}{\bb{R}}\nc{\bbS}{\bb{S}}\nc{\bbT}{\bb{T}}
\nc{\bbU}{\bb{U}}\nc{\bbV}{\bb{V}}\nc{\bbW}{\bb{W}}\nc{\bbX}{\bb{X}}
\nc{\bbY}{\bb{Y}}\nc{\bbZ}{\bb{Z}}
\nc{\mbf}[1]{{\mathbf #1}}
\nc{\bfA}{\mbf{A}}\nc{\bfB}{\mbf{B}}\nc{\bfC}{\mbf{C}}\nc{\bfD}{\mbf{D}}
\nc{\bfE}{\mbf{E}}\nc{\bfF}{\mbf{F}}\nc{\bfG}{\mbf{G}}\nc{\bfH}{\mbf{H}}
\nc{\bfI}{\mbf{I}}\nc{\bfJ}{\mbf{J}}\nc{\bfK}{\mbf{K}}\nc{\bfL}{\mbf{L}}
\nc{\bfM}{\mbf{M}}\nc{\bfN}{\mbf{N}}\nc{\bfO}{\mbf{O}}\nc{\bfP}{\mbf{P}}
\nc{\bfQ}{\mbf{Q}}\nc{\bfR}{\mbf{R}}\nc{\bfS}{\mbf{S}}\nc{\bfT}{\mbf{T}}
\nc{\bfU}{\mbf{U}}\nc{\bfV}{\mbf{V}}\nc{\bfW}{\mbf{W}}\nc{\bfX}{\mbf{X}}
\nc{\bfY}{\mbf{Y}}\nc{\bfZ}{\mbf{Z}}
\nc{\bfa}{\mbf{a}}\nc{\bfb}{\mbf{b}}\nc{\bfc}{\mbf{c}}\nc{\bfd}{\mbf{d}}
\nc{\bfe}{\mbf{e}}\nc{\bff}{\mbf{f}}\nc{\bfg}{\mbf{g}}\nc{\bfh}{\mbf{h}}
\nc{\bfi}{\mbf{i}}\nc{\bfj}{\mbf{j}}\nc{\bfk}{\mbf{k}}\nc{\bfl}{\mbf{l}}
\nc{\bfm}{\mbf{m}}\nc{\bfn}{\mbf{n}}\nc{\bfo}{\mbf{o}}\nc{\bfp}{\mbf{p}}
\nc{\bfq}{\mbf{q}}\nc{\bfr}{\mbf{r}}\nc{\bfs}{\mbf{s}}\nc{\bft}{\mbf{t}}
\nc{\bfu}{\mbf{u}}\nc{\bfv}{\mbf{v}}\nc{\bfw}{\mbf{w}}\nc{\bfx}{\mbf{x}}
\nc{\bfy}{\mbf{y}}\nc{\bfz}{\mbf{z}}

\nc{\mcal}[1]{{\mathcal #1}}
\nc{\calA}{\mcal{A}}\nc{\calB}{\mcal{B}}\nc{\calC}{\mcal{C}}\nc{\calD}{\mcal{D}}
\nc{\calE}{\mcal{E}} \nc{\calF}{\mcal{F}}\nc{\calG}{\mcal{G}}\nc{\calH}{\mcal{H}}
\nc{\calI}{\mcal{I}}\nc{\calJ}{\mcal{J}}\nc{\calK}{\mcal{K}}\nc{\calL}{\mcal{L}}
\nc{\calM}{\mcal{M}}\nc{\calN}{\mcal{N}}\nc{\calO}{\mcal{O}}\nc{\calP}{\mcal{P}}
\nc{\calQ}{\mcal{Q}}\nc{\calR}{\mcal{R}}\nc{\calS}{\mcal{S}}\nc{\calT}{\mcal{T}}
\nc{\calU}{\mcal{U}}\nc{\calV}{\mcal{V}}\nc{\calW}{\mcal{W}}\nc{\calX}{\mcal{X}}
\nc{\calY}{\mcal{Y}}\nc{\calZ}{\mcal{Z}}
\nc{\fA}{\frak{A}}\nc{\fB}{\frak{B}}\nc{\fC}{\frak{C}} \nc{\fD}{\frak{D}}
\nc{\fE}{\frak{E}}\nc{\fF}{\frak{F}}\nc{\fG}{\frak{G}}\nc{\fH}{\frak{H}}
\nc{\fI}{\frak{I}}\nc{\fJ}{\frak{J}}\nc{\fK}{\frak{K}}\nc{\fL}{\frak{L}}
\nc{\fM}{\frak{M}}\nc{\fN}{\frak{N}}\nc{\fO}{\frak{O}}\nc{\fP}{\frak{P}}
\nc{\fQ}{\frak{Q}}\nc{\fR}{\frak{R}}\nc{\fS}{\frak{S}}\nc{\fT}{\frak{T}}
\nc{\fU}{\frak{U}}\nc{\fV}{\frak{V}}\nc{\fW}{\frak{W}}\nc{\fX}{\frak{X}}
\nc{\fY}{\frak{Y}}\nc{\fZ}{\frak{Z}}
\nc{\fa}{\frak{a}}\nc{\fb}{\frak{b}}\nc{\fc}{\frak{c}} \nc{\fd}{\frak{d}}
\nc{\fe}{\frak{e}}\nc{\fFf}{\frak{f}}\nc{\fg}{\frak{g}}\nc{\fh}{\frak{h}}
\nc{\fri}{\frak{i}}\nc{\fj}{\frak{j}}\nc{\fk}{\frak{k}}\nc{\fl}{\frak{l}}
\nc{\fm}{\frak{m}}\nc{\fn}{\frak{n}}\nc{\fo}{\frak{o}}\nc{\fp}{\frak{p}}
\nc{\fq}{\frak{q}}\nc{\fr}{\frak{r}}\nc{\fs}{\frak{s}}\nc{\ft}{\frak{t}}
\nc{\fu}{\frak{u}}\nc{\fv}{\frak{v}}\nc{\fw}{\frak{w}}\nc{\fx}{\frak{x}}
\nc{\fy}{\frak{y}}\nc{\fz}{\frak{z}}

\newtheorem{theorem}{Theorem}[section]
\newtheorem{lemma}[theorem]{Lemma}
\newtheorem{corollary}[theorem]{Corollary}
\newtheorem{prop}[theorem]{Proposition}

\theoremstyle{definition}
\newtheorem{definition}[theorem]{Definition}
\newtheorem{example}[theorem]{Example}
\newtheorem{remark}[theorem]{Remark}

 \DeclareMathOperator{\id}{id}
 \DeclareMathOperator{\Sym}{Sym}

\DeclareMathOperator{\Hom}{{Hom}}

\DeclareMathOperator{\Gr}{Gr}

\DeclareMathOperator{\SL}{SL}

\newcommand{\cQ}{\mathcal {Q}}

\newcommand{\pt}{\text{pt}}

\newcommand{\fraks}{\mathfrak{s}}

\newcommand{\al}{\alpha}
\newcommand{\la}{\lambda}

\newcommand{\ka}{\kappa}

\newcommand{\de}{\delta}

\newcommand{\be}{\beta}

\DeclareMathOperator{\Kr}{{Kr}}

\newcommand{\unit}{\mathbf{1}}

\newcommand{\ep}{\epsilon}
\newcommand{\hatX}{\hat{X}}
\newcommand{\hatQ}{{\hat{\cQ}}}
\newcommand{\hatS}{{\hat{S}}}

\newcommand{\hatD}{{\hat{\bfD}}}

\DeclareMathOperator{\pr}{pr}

\newcommand{\frakp}{\mathfrak{p}}
\newcommand{\frakq}{\mathfrak{q}}

\newcommand{\aff}{{\operatorname{a}}}
\newcommand{\inv}{{\operatorname{inv}}}

\newcommand{\frakX}{\mathfrak{X}}

\newcommand{\frakd}{\mathfrak{d}}

\newcount\cols
{\catcode`,=\active\catcode`|=\active
 \gdef\Young(#1){\hbox{$\vcenter
 {\mathcode`,="8000\mathcode`|="8000
  \def,{\global\advance\cols by 1 &}%
  \def|{\cr
        \multispan{\the\cols}\hrulefill\cr
        &\global\cols=2 }%
  \offinterlineskip\everycr{}\tabskip=0pt
  \dimen0=\ht\strutbox \advance\dimen0 by \dp\strutbox
  \halign
   {\vrule height \ht\strutbox depth \dp\strutbox##
    &&\hbox to \dimen0{\hss$##$\hss}\vrule\cr
    \noalign{\hrule}&\global\cols=2 #1\crcr
    \multispan{\the\cols}\hrulefill\cr%
   }
 }$}}
}

\title[Affine Grassmannian]
{Equivariant oriented homology of the affine Grassmannian}

\author[C.~Zhong]{Changlong~Zhong}
\address{State University of New York at Albany, 1400 Washington Ave, CK399, Albany, NY, 12222}
\email{czhong@albany.edu}

\date{}
\begin{document}
\maketitle
\begin{abstract}We generalize the property of small-torus equivariant K-homology of the affine Grassmannian to general oriented (co)homology theory in the sense of Levine and Morel. The main tool we use is the formal affine Demazure algebra associated to the affine root system. More precisely, we prove that the small-torus equivariant oriented cohomology of the affine Grassmannian satisfies the GKM condition. We also show that its dual, the small-torus equivariant homology, is isomorphic to the centralizer of the equivariant oriented cohomology of a point in the  the formal affine Demazure algebra. 
\end{abstract}

\section{Introduction}
Let $h$ be an oriented cohomology theory in the sense of Levine and Morel. Let $G$ be a semi-simple linear algebraic group over $\bbC$ with maximal torus $T$ and a Borel subgroup $B$. Let $\Gr_G$ be the affine Grassmannian of $G$. $T$ is called the small torus, in contrary to the big torus $T_\aff$ of $\Gr_G$. The theory of $h_{T_\aff}(\Gr_G)$ when $h$ is the equivariant cohomology or the K-theory, is studied by Kostant and Kumar in \cite{KK86, KK90}. It is dual to the so-called affine nil-Hecke algebra (equivariant cohomology case) or the affine 0-Hecke algebra. Alternatively, the affine nil-Hecke algebra and the affine 0-Hecke algebra can be called the equivariant  homology and the equivariant K-homology theory.

The small torus equivariant  homology theory  $H_T(\Gr_G)$ of the affine Grassmannian was first studied by Peterson \cite{P97}. Moreover, he raised a conjecture (without a proof) saying that $H_T(\Gr_G)$ is isomorphic to the quantum cohomology $QH_T(G/B)$of $G/B$. This conjecture, together with its partial flag variety version, is proved by Lam-Shimozono in \cite{LS10}. One key step is the identification of $H_T(G/B)$ with the centralizer of $H_T(\pt)$ in $H_T(G_\aff/B_\aff)$ where $G_\aff$ is the Kac-Moody group associated to the affine root system and $B_\aff$ is its Borel subgroup. 

For K-theory, similar property was expected to hold. In \cite{LSS10}, the authors study the K-theoretic Peterson subalgebra, i.e., the centralizer of the ring $K_T(\pt)$ in the small-torus affine 0-Hecke algebra, i.e., the equivariant K-homology $K_T(G_\aff/B_\aff)$. It is proved that this algebra is isomorphic to $K_T(\Gr_G)$. One of the main tools is the small-torus GKM condition of the $T$-equivariant K-cohomology.  In \cite{LLMS18}, some evidence was provided in supporting the K-theory Peterson Conjecture. In \cite{K18}, using the study of semi-infinite flag variety, Kato proves this conjecture. More precisely, he embeds quantum K-theory of flag variety and certain localization of the Peterson subalgebra into $T$-equivariant K-theory of semi-infinite flag variety, and proves that their image coincide. 

In all the work mentioned above, the Peterson subalgebra plays key roles. In this paper, we generalize the construction of the Peterson subalgebra into general oriented cohomology theory $h$. Associated to such theory, there is a formal group law $F$ over the coefficient ring $R=h(\pt)$. Associated to $F$ and a Kac-Moody root system, in \cite{CZZ1, CZZ2, CZZ3, CZZ4}, the author generalized Kostant-Kumar's construction and defined the formal affine Demazure algebra (FADA). It is a non-commutative algebra generated by the divided difference operators. Its dual give an algebraic model for $h_{T_\aff}(G_\aff/B_\aff)$. Since Levine-Morel's oriented cohomology theory is only defined for smooth projective varieties, in this paper we do not intend to generalize the geometric theory. Instead, we only work with the algebraic model, i.e., the FADA associated to $h$. 

Following the same idea as the work mentioned above in cohomology and K-theory,  we look at the small-torus (the torus $T$) version, which is very similar as the big torus case $T_\aff$. We define the small torus FADA, $\bfD_{W_\aff}$.   In this paper, our first main result (Theorem \ref{thm:GKM}) shows that  the algebraic models for $h_T(G_\aff/B_\aff)$ and $h_T(\Gr_G)$, i.e.,  $\bfD^*_{W_\aff}$ and $(\bfD^*_{W_\aff})^W$, satisfy the small torus GKM condition. Based on that, we prove the second main result (Theorem \ref{thm:Peterson}), which shows that the dual of $h_T(\Gr_G)$, denoted by $\bfD_{Q^\vee}$ ($Q^\vee$ being the coroot lattice), coincides with the centralizer of $h_T(\pt)$ in the FADA $\bfD_{W_\aff}$. This defines the Peterson subalgebra associated to $h$.  

Our result generalizes and extends properties for equivariant cohomology and K-theory. Moreover, our method is uniform and does not reply on the specific oriented cohomology theory.  As an application of this construction, we define actions of the FADA (of the  big and small torus) on the algebraic models fo $h_{T_\aff}(\Gr_G)$ and $h_T(\Gr_G)$. This is called the left Hecke action. For finite flag varieties case it is studied in \cite{MNS22} using geometric arguments (see also \cite{B97,K03,T09, LZZ20}). For connective K-theory (which specializes to cohomology and K-theory), we compute the recursive formulas for certain basis in $h_T(\Gr_G)$ (Theorem \ref{thm:con}). 

It is natural to consider generalizing Kato's construction to this case, that is, invert  Schubert classes in $\bfD_{Q^\vee}$ corresponding to $t_\la\in Q^\vee_{<}$. This localization for K-theory was proved to be isomorphic  to $QK_T(G/B)$. For $h$ beyong singular cohomology and K-theory, however, the first obstruction is that there is no `quantum' oriented cohomology theory defined. The other obstruction is that  the divided difference operators do not satisfy braid relations. This was a key step in Kato's construction (see \cite[Theorem 1.7]{K18}). The author plans to investigate this in a future paper. 

This paper is organized as follows: In \S1 we recall the construction of the FADA for the big torus $T_\aff$, and in \S2 we compute the recursive formulas via the left Hecke action. In \S3 we we repeat the construction for the small torus and indicates the difference from the big torus case. In \S4 we prove that dual of the small torus FADA satisfies the small torus GKM condition, and in \S5 we define the Peterson subalgebra and  show that it coincides with the centralizer of $h_T(\pt)$.  In the appendix we provide some computational result in the $\hat A_1$ case.

\subsection*{Notations} 
Let $G\supset B\supset T$ be such that $G$ is simple, simply connected algebraic group over $\bbC$ with a Borel subgroup $B$ and a torus $T$. Let $G_\aff\supset B_\aff\supset T_\aff $ where $G_\aff$ is the affine Kac-Moody group with Borel subgroup $B_\aff$ and the affine torus $T_\aff$. Let $P$ be the maximal parabolic group scheme  so that $G_\aff/P=\Gr_G$ is the affine Grassmannian. Let $T^*$ (resp. $T_\aff^*$) be the group of characters of $T$ (resp. $T_\aff$), then $T^*_\aff=T^*\oplus\bbZ\de$. 

 Let $W$ be the Weyl group of $G$,  $I=\{\al_1,...,\al_n\}$ be the simple roots, $Q=\oplus _{i}\bbZ \al_i\subset T^*$ be the root lattice, $Q^\vee=\oplus_i \bbZ\al_i^\vee$ be the coroot lattice,  $\theta$ be the longest element, $\delta$ is the null root, $\al_0=-\theta+\de$ be the extra simple root. Denote  $I_\aff=\{\al_0,...,\al_n\}$.  For each $\la\in Q^\vee$, let $t_\la$ be the translation acting on $Q$. We then have $t_{\la_1} t_{\la_2}=t_{\la_1+\la_2}$, and $wt_\la w^{-1}=t_{w(\la)}, w\in W$. Let
$Q^\vee_{\leq}$ be the set of antidominant coroots, $Q^\vee_{<}$ be the set of strictly antidominant coroots (i.e., $(\la, \al_i)<0~  \forall i\in I$). 
Let  $W_\aff=W\ltimes Q^\vee$ be the affine Weyl group,  $\ell$ be the length function on $W_\aff$, and $w_0\in W$ be the longest element. 

Let  $\Phi$ be the set of roots for $W$,  $\Phi_\aff=\bbZ\de+\Phi$ be the set of real affine  roots, and $\Phi^\pm_\aff, \Phi^\pm$ be the corresponding set of positive/negative roots for the corresponding systems. Let $\inv(w)=w^{-1}\Phi_\aff^+\cap \Phi^-_\aff$. We have
\[
\Phi_\aff^+=\{\al+k\de|\al\in \Phi^+, k= 0 \text{ or }\al\in \Phi, k>0\}.
\]
 Let  $W_\aff^-$ be the minimal length representatives of $W_\aff/W$. There is a bijection 
\[
W_\aff^-\to Q^\vee, \quad  w\mapsto \la, \text{ if }wW=t_\la W.\]
Moreover, $W_\aff^-\cap Q^\vee=\{t_\la|\la\in Q_{\le}\}. $
The action of $\al+k\de$ on $\mu+m\de\in Q\oplus \bbZ\de$ is given by
\[
s_{\al+k\de}(\mu+m\de)=\mu+m\de-\langle\mu, \al^\vee\rangle(\al+k\de).
\]
In particular, for $\la\in Q^\vee, w\in W, \mu\in Q$, we have $s_{\al+k\de}=s_\al t_{k\al^\vee}, ~wt_\la(\mu)=w(\mu). $

We say the set of reduced sequences $I_w, w\in W_\aff$ is $W$-compatible if $I_w=I_{u}\cup I_v$ for $w=uv, u\in W_\aff^-, v\in W$. 
\section{FADA for the big torus}

In this section, we recall the construction of the formal affine Demazure algebra (FADA) for the affine root system. All the construction can be found in \cite{CZZ4}.

\subsection{} Let $F$ be a one dimensional formal group law over  a  domain $R$ with characteristic 0. 
Following from \cite{LM07} that there is an oriented cohomology $h$ whose associated formal group law is $F$. In this paper we won't need any geometric property of this $h$, since our treatment is pure algebraic and self-contained.
\begin{example}
Let $F=F_\fc=x+y-\fc xy$ be the connective formal group law (for connective K-theory) over $R=\bbZ[\fc]$. Specializing to $\fc=0$ or $\fc=1$, one obtains the additive or multiplicative formal group law. One of the simplest formal group laws beyond $F_\fc$ is the  hyperbolic formal group law considered in \cite{LZZ20}:
\[
F(x,y)=\frac{x+y-\fc xy}{1+a xy}, ~R=\bbZ[\fc, a].
\]
\end{example}

Let $\hatS$ be the formal group algebra of $T_\aff^*$ defined in \cite{CPZ13}.  That is, 
\[
\hatS=R[[x_\mu|\mu\in T_\aff^*]]/\calJ_F,  
\]
where $\calJ_F$ is the closure of the ideal generated $x_0$ and $x_{\mu_1+\mu_2}-F({x_{\mu_1}, x_{\mu_2}}), \mu_1, \mu_2\in T^*_\aff$. Indeed, after fixing a basis of $T^*_\aff\cong \bbZ^{n+1}$,  $\hatS$ is isomorphic to the power series ring $R[[x_1,...,x_{n+1}]]$. 
\begin{remark} \label{rm:poly} If $F=F_\mathfrak{c}$ is the connective formal group law, one can just replace $\hatS$ by $R[x_\mu|\mu\in T^*_\aff]/\calJ_F$. In other words, in this case one can use the polynomial ring instead of the power series ring. For instance, if $\fc=0$, then $\hatS\cong\Sym_R(T^*_\aff), x_\mu\mapsto \mu$. If $\fc\in R^\times$, then $\hatS\cong R[T^*_\aff], x_\mu\mapsto \fc^{-1}(1-e^{-\mu})$. Throughout this paper, whenever we specializes to $F_\fc$, we assume that $\hatS$ is the polynomial version. 
\end{remark}

\subsection{}
 Define $\hat \cQ=\hatS [\frac{1}{x_{\al}}, \al\in \Phi_\aff]. $
The Weyl groups $W_\aff$ acts on $\hatQ$, so we can define the twisted group algebra
$
\hatQ_{W_\aff}:=\hatQ\rtimes R[W_\aff],
$, which is a free left $\hatQ$-module with basis denoted by $\eta_w, w\in W_\aff$ and the product $c\eta_wc'\eta_{w'}=cw(c')\eta_{ww'},~ c,c'\in \hatQ$.

For each $\al\in \Phi_\aff$, define $\kappa_{\al}=\frac{1}{x_{\al}}+\frac{1}{x_{-\al}}\in \hatS$. If $F=F_\fc$, then $\kappa_\al=\fc$. For each simple root $\al_i$, we  define the Demazure element $\hatX_{\al_i}=\frac{1}{x_{\al_i}}(1-\eta_{s_i})$.  It is easy to check that $\hatX_{\al}^2=\kappa_\al \hatX_\al.$
For simplicity, denote $\eta_i=\eta_{\al_i}=\eta_{s_i}$, $x_{\pm i}=x_{\pm \al_i}$,  $\hatX_i=\hatX_{\al_i}$, $i\in I_\aff$. 
If $I_w=(i_1,...,i_k), i_j\in I_\aff$ is a reduced sequence of $w\in W_\aff$, we define $\hatX_{I_w}$ correspondingly. It is well known that they depends on the choice of $I_w$, unless $F=F_\fc$. 

Write 
\begin{equation}\label{eq:basis}\hatX_{I_w}=\sum_{v\le w}\hat a_{I_w,v}\eta_v, \quad \eta_w=\sum_{v\le w}\hat b_{w,I_v}\hat X_{I_v}, \quad \hat a_{I_w,v}\in \hatQ,~ \hat b_{w,I_v}\in \hatS,
\end{equation}
 then we have $\hat b_{w,I_w}=\prod_{\al\in \inv(w)}x_\al=\frac{1}{\hat a_{I_w, w}}$.

Let $\hatD_{W_\aff}$ be the subalgebra of $\hatQ_{W_\aff}$ generated by $\hatS$ and $\hatX_i, i\in I_\aff$. This is called the formal affine Demazure algebra (FADA) for the big torus.  It is easy to see that $\hatX_{I_w}, w\in W_\aff$ is a $\hatQ$-basis of $\hatQ_{W_\aff}$,  and it is proved in \cite{CZZ4} that  it is also a basis of the  left $\hatS$-module $\hatD_{W_\aff}$. Note that $W\subset \hatD_{W_\aff}$ via the map $s_i\mapsto \eta_{i}=1-x_i\hatX_i\in \hatD_{W_\aff}$.

\begin{remark}It is not difficult to derive that there is a residue description of the coefficients in the expression of  elements of $\hatD_{W_\aff}$ as linear combinations of $\eta_w$. Such description was first given in \cite{GKV97}. See \cite{ZZ17} for more details. 
\end{remark}

\subsection{}
We define the duals of left modules:
\[
\hatQ_{W_\aff}^*=\Hom_\hatQ(\hatQ_{W_\aff}, \hatQ)=\Hom (W_\aff, \hatQ), \quad \hatD_{W_\aff}^*=\Hom_\hatS(\hatD_{W_\aff}, \hatS).
\]
Dual to the elements $\eta_w, \hatX_{I_w}\in \hatD_{W_\aff}\subset \hatQ_{W_\aff}$, we have $\hat f_w, \hatX_{I_w}^*\in \hatD_{W_\aff}^*\subset \hatQ_{W_\aff}^*$. The product structure on $\hatQ_{W_\aff}^*$ is defined by $\hat f_w\hat f_v=\de_{w,v}\hat f_w,$ with the unit  given by $\unit=\prod_{w\in W_\aff}\hat f_w$. Note that here we usually use $\prod$ to denote a sum of  (possibly) infinitely many terms, and $\sum$ to denote a finite sum. 
\begin{lemma}\label{lem:bigtorus}We have 
\[
\hatD_{W_\aff}^*=\{\hat f\in \hatQ_{W_\aff}^*|\hat f(\hatD_{W_\aff})\subset \hatS \}.
\]
\end{lemma}
\begin{proof} Denote the RHS by $\calZ_1$. 
It is clear that $\hatD_{W_\aff}^*$ is contained in $\calZ_1$ since $\hatX_{I_v}$ generate $\hatD_{W_\aff}$,   $\hatX_{I_w}^*$ generate $\hatD_{W_\aff}^*$, and  $\hatX_{I_w}^*(\hatX_{I_v})=\de_{w,v}$.  Conversely, let $\hat f=\prod_{\ell(w)\ge k}c_w\hat f_w\in \calZ_1$. If $\ell(u)=k$, then from \eqref{eq:basis}, we have
\[
\hat f(\hatX_{I_u})=\prod_{\ell(w)\ge k}c_w\hat f_w(\sum_{v\le u}\hat a_{I_u,v}\eta_v)=c_u\hat a_{I_u,u}\in \hatS.
\]
Denote  $\hat f':=\hat f-\sum_{\ell(u)=k}c_u\hat a_{I_u,u}\hat X_{I_u}^*$. Note that $\hat X_{I_u}^*=\prod_{w\in W_\aff}\hat b_{w,I_u}f_w$ and $\hat b_{u,I_u}\hat a_{I_u,u}=1$, so  for any $u$ with $\ell(u)=k$, we have $\hat f'(\eta_u)=c_u-c_u\hat a_{I_u,u}\hat X_{I_u}^*(\eta_u)=c_u-c_u=0$,  so $\hat f'$ is a linear combination of $\hat f_w,\ell(w)\ge k+1$. Repeating this process, we get that  $\hat f \in \hatD_{W_\aff}^*$. 
\end{proof}
\subsection{}
There is an $\hatQ$-linear action of $\hatQ_{W_\aff}$ on $\hatQ_{W_\aff}^*$, defined by 
\[
(z\bullet \hat f)(z')=\hat f(z'z), \quad z,z'\in \hatQ_{W_\aff}, \hat f\in \hatQ_{W_\aff}^*.
\]
This is called the right Hecke action. 
We have 
\[c\eta_w\bullet c'\hat f_{w'}=c'w'w^{-1}(c)\hat f_{w'w^{-1}}, ~c, c'
\in \hatQ.\]
 It follows from Lemma \ref{lem:bigtorus} and similar reason as in  \cite[\S10]{CZZ2} that this induces  an action of $\hatD_{W_\aff}$ on $\hatD_{W_\aff}^*$. Moreover, it induces an action of $W\subset \hatD_{W_\aff}$ on $\hatQ_{W_\aff}^*$ and $\hatD_{W_\aff}^*$. By definition it is easy to get
\begin{eqnarray}\label{eq:bullet}
\hatX_\al\bullet \prod_{w\in W_\aff}c_w\hat f_w=\prod_{w\in W_\aff}\frac{c_w-c_{s_{w(\al)}w}}{x_{w(\al)}}\hat f_w.
\end{eqnarray}
 The following proposition is proved in the  finite case in \cite[Lemma 10.2, Theorem 10.7]{CZZ2}. 
\begin{prop}\label{prop:bigtorusGKM}
The subset $\hatD_{W_\aff}^*\subset\hatQ_{W_\aff}^*$ satisfies the following   (big-torus) GKM condition:
\begin{eqnarray*}
\hatD_{W_\aff}^*=\{\hat f\in \hatQ_{W_\aff}^*|\hat f (\eta_w)\in \hatS\text{ and }\hat f(\eta_w-\eta_{s_{\al}w})\in x_{\al} \hatS, \forall \al\in \Phi_\aff\}.
\end{eqnarray*}
\end{prop}

\begin{proof}
Denote the RHS by $\calZ_2$.  Let $\hat f \in \hatD_{W_\aff}^*$, we know  $\hat X_\al\bullet\hat f\in \hatD_{W_\aff}^*$. Then \eqref{eq:bullet} implies that $\hat f$ satisfies the condition defining $\calZ_2$, so $\hatD_{W_\aff}^*\subset \calZ_2$. 

For the other direction, we first show that $\hatD^*_{W_\aff}$ is a maximal $\hatD_{W_\aff}$-submodule of $\hatS_{W_\aff}^*:=\Hom(W_\aff, \hatS)$. This can be proved as follows: if $M\subset \hatS_{W_\aff}^*$ is a $\hatD_{W_\aff}$-module, for any $\hat f\in M$, we have $\hatX_I\bullet\hat  f\in M\subset \hatS_{W_\aff}^*$, so $\hat X_I\bullet \hat f(\eta_e)=\hat f(\hat X_I)\in \hatS$, so $f\in \hatD_{W_\aff}^*$. 
One then can   show that the subset $\calZ_2$ is a $\hatD_{W_\aff}$-module, which follows from the same proof as in the finite case in \cite[Theorem 10.2]{CZZ2}. Since  $\hatD_{W_\aff}^*$ is a maximal submodule, we have $\calZ_2\subset \hatD_{W_\aff}^*$. The proof is finished.
\end{proof}

\subsection{}
We can similarly define the non-commutative ring $\hatQ_{Q^\vee}=\hatQ\rtimes R[Q^\vee]$ with a $\hatQ$-basis $\eta_{t_\la}, \la\in Q^\vee$. Then there is a canonical map of left $\hatQ$-modules:
\[
\pr:\hatQ_{W_\aff}\to \hatQ_{Q^\vee}, \quad c\eta_{t_\la w}\mapsto c\eta_{t_{\la}}, \quad w\in W, \la\in Q^\vee, c\in \hatQ.
\]
Define $\hatD_{W_\aff/W}=\pr(\hatD_{W_\aff})\subset \hatQ_{Q^\vee}$. Indeed, this is the same as the relative Demazure module defined in \cite[\S11]{CZZ2}.

We can also consider the $\hatQ$-dual $\hatQ_{Q^\vee}^*$ and the $\hatS$-dual $\hatD_{W_\aff/W}^*$. The elements dual to $\eta_{t_\la}\in \hatQ_{Q^\vee}$ are denoted by $\hat f_{t_\la}$. The projection $\pr$ then induces embeddings $\pr^*:\hatQ_{Q^\vee}^*\hookrightarrow \hatQ_{W_\aff}^*$ and $\pr^*:\hatD_{W_\aff/W}\hookrightarrow \hatD_{W_\aff}^*$. It is easy to see that 
\[\pr^*(\hat f_{t_\la})=\sum_{v\in W}\hat f_{t_\la v}.\]
Moreover, similar as in the finite case \cite[Lemma 11.7]{CZZ2}, we have 
\[
\pr^*(\hatQ_{Q^\vee}^*)=(\hatQ_{W_\aff}^*)^W, \quad \pr^*(\hatD_{W_\aff/W}^*)=(\hatD_{W_\aff}^*)^W.
\]
Indeed,  elements of $\pr^*(\hatQ_{Q^\vee}^*)=(\hatQ_{W_\aff}^*)^W$ are precisely the elements $\hat f\in \hatQ_{W_\aff}^*$ satisfying $\hat f(\eta_{t_\la w}-\eta_{t_\la })=0$ for any $\la\in Q^\vee, w\in W$. It follows from similar reason as \cite[Corollary 8.5, Lemma 11.5]{CZZ2} that if $I_w, w\in W_\aff$ is $W$-compatible, then  $\hat b_{{uv},I_w}=\hat b_{u,I_w}$ for any $v\in W$. We then have
\begin{lemma} \label{lem:bigWinv}Assume  the sequences $I_w, w\in W_\aff$ is $W$-compatible,  then $\pr(X_{I_w}), w\in W_\aff^-$ is a basis of $\hatD_{W_\aff/W}$, and 
$\{\hatX_{I_w}^*, w\in W_\aff^-\}$ is a $\hatQ$-basis of $(\hatQ_{W_\aff}^*)^W$ and a $\hatS$-basis of $(\hatD_{W_\aff}^*)^W$. 
\end{lemma}
 Note that $(\hatD_{W_\aff}^*)^W$ is the algebraic model for $h_{T_\aff}(\Gr_G)$ and the embedding $(\hatD_{W_\aff}^*)^W\subset \hatD_{W_\aff}^*$ is the algebraic model for the pull-back $h_{T_\aff}(\Gr_G)\to h_{T_\aff}(G_\aff/B_\aff)$.

\subsection{}
Similar as the finite case in \cite[\S3]{LZZ20}, there is another action of $\hatQ_{W_\aff}$  on $\hatQ_{W_\aff}^*$ by 
\[
a\eta_v\odot b\hat f_w=av(b)\hat f_{vw}, ~a,b\in \hatQ, w,v\in W_\aff.
\]
This is called the left Hecke action. It is easy to see that it commutes with the $\bullet$-action.  Note however that the $\odot$-action is not $\hatQ$-linear. 
\begin{lemma}The $\odot $ action of $\hatQ_{W_\aff}$ on $\hatQ_{W_\aff}^*$ induces an action of $\hatD_{W_\aff}$ on $\hatD_{W_\aff}^*$.
\end{lemma}
\begin{proof}
We have 
\[\hat X_{\al}\odot \prod_wc_w\hat f_w=\frac{1}{x_\al}(1-\eta_\al)\odot \prod_wc_w\hat f_w=\prod_{w}\frac{c_w-s_\al(c_{s_\al w})}{x_\al}\hat f_w.\]
let $d_{w,\al}=\frac{c_w-s_\al(c_{s_\al w})}{x_\al}$. We show that $d_{w,\al}$ satisfy the big-torus GKM condition, that is, $d_{w,\al}-d_{s_{\be} w, \al}\in x_{\be} \hatS$  for any $\be$.

Denote $c_w-c_{s_\al w}=x_\al p, p\in \hat S$ and $x_{-\al}=-x_\al+x_\al^2 q, q\in \hat S$. If $\be=\al$, then we have 
\begin{eqnarray*}
d_{w,\al}-d_{s_\be w, \al}&=&\frac{c_w-s_\al(c_{s_\al w})-c_{s_\al w}+s_\al (c_w)}{x_\al}=\frac{x_\al p+s_\al (c_w)-s_\al(c_w-x_\al p)}{x_\al}\\
&=&p+\frac{x_{-\al}s_\al(p)}{x_\al}
=p-s_\al(p)+x_\al q,
\end{eqnarray*}
which is clearly a multiple of $x_\al$. If $\be\neq \al$, then 
\[
d_{w,\al}-d_{s_\be w,\al}=\frac{c_w-s_\al(c_{s_\al w})-(c_{s_\be w}-s_\al(c_{s_\al s_\be w}))}{x_\al}=\frac{c_w-s_\al(c_{s_\al w})-c_{s_\be w}+s_\al(c_{s_\al s_\be w})}{x_\al}.
\]
Since $x_\al, x_\be$ are coprime \cite[Lemma 2.2]{CZZ4}, it suffices to prove the numerator is divisible by $x_\be$. Note $c_w-c_{s_\be w}$ is already divisible by $x_\be$. Furthermore,  $c_w-c_{s_{\al}s_\be w}=c_w-c_{s_{s_\al(\be)}s_\al w}$,  so it is divisible by $s_{s_\al(\be)}$. Therefore, $-s_\al(c_{s_\al w})+s_\al(c_{s_\al s_\be w})$ is divisible by $s_\al(x_{s_\al(\be)})=x_\be$. The proof is finished. 
\end{proof}
Consequently, the $\odot$-action of $\hatD_{W_\aff}$ on $\hatD_{W_\aff}^*$ restricts to an action on $(\hatD_{W_\aff}^*)^W$. 
\subsection{}\label{subsec:char}

 Indeed, there is a  characteristic map 
\[
\mathbf{c}:\hatS\to \hatD^*_{W_\aff}, ~z\mapsto z\bullet \unit,
\]
whose geometric model is  the map sending a character of the torus to the first Chern class of the associated line bundle over the flag variety \cite[\S10]{CZZ3}. 
We then have a map
\[
\phi:\hatS\otimes_{\hatS^{W_\aff}}\hatS\to \hatD_{W_\aff}^*, ~a\otimes b\mapsto a\bfc(b)=\prod_w a w(b)\hat f_w.
\]
This is proved to be an isomorphism in some cases. 
It is easy to see that for any $z\in \bfD_{W_\aff}$,  there are the following commutative diagrams
\[
\xymatrix{\hatS\otimes_{\hatS^{W_\aff}}\hatS\ar[r]^-\phi\ar[d]^-{z\cdot\_\otimes \id}&\hatD_{W_\aff}^*\ar[d]^-{z\odot\_}\\
\hatS\otimes_{\hatS^{W_\aff}}\hatS\ar[r]^-\phi & \hatD_{W_\aff}^*}, \qquad 
\xymatrix{\hatS\otimes_{\hatS^{W_\aff}}\hatS\ar[r]^-\phi\ar[d]^-{\id\otimes z\cdot\_}& \hatD^*_{W_\aff}\ar[d]^-{z\bullet\_}\\
\hatS\otimes_{\hatS^{W_\aff}}\hatS\ar[r]^-\phi & \hatD_{W_\aff}^*}. 
\]

\section{Equivariant connective K-theory of the affine Grassmannian} \label{sec:connective}
As an application of the left Hecke action, we derive the  recursive formulas for this action on  bases in connective K-theory of $\Gr_G$.  In this section only, assume $F=F_\fc$. Our results specialize to equivariant K-theory (resp. equivariant cohomology) by letting $\fc=1$ (resp. $\fc=0$). In both cases, our results are only known for  flag varieties of finite root systems. Since $\hatX_i$ do not satisfy the braid relations, the result of this section do not generalize to general $F$.

\subsection{} Denote $\ep_w=(-1)^{\ell(w)}$ and  $\fc_w=c^{\ell(w)}$. We have  $x_{-\al}=\frac{x_\al}{\fc x_\al-1}$ and $\kappa_\al=\fc$ for any $\al$, and $\hatX_{I_w}$ can be denoted by $\hatX_w$. 

 Note that there is another operator $\hat Y_i=\hat Y_{\al_i}=\fc-\hatX_{\al_i}$ such that $\hat Y_{\al_i}^2=\fc \hat Y_{\al_i}$ and braid relations are satisfied.  This is the algebraic model of the composition $h_{T_\aff}(G_\aff/B_\aff)\to h_{T_\aff}(G_\aff/P_i)\to  h_{T_\aff}(G_\aff/B_\aff)$ where $P_i$ is the minimal parabolic subgroup corresponding to $\al_i\in I_\aff$. Moreover, we have 
\[\hat X_w=\sum_{v\le w}\ep_v\fc_w\fc_v^{-1}\hat Y_v.\] Most  properties of $\hatX_w$ are also satisfied by $\hat Y_w$, except for Lemma \ref{lem:bigWinv}. Indeed,  $\hat Y_w^*, w\in W_\aff^-$ is not $W$-invariant. 

Denote $x_\Phi=\prod_{\al\in \Phi^-}x_\al$. It is well known that $Y_{w_0}=\sum_{w\in W}\eta_w\frac{1}{x_\Phi}.$
 Moreover, the map $Y_{w_0}\bullet\_:\hatD_{W_\aff}^*\to (\hatD_{W_\aff}^*)^W$ is the algebraic model for the map $h_{T_\aff}(G_\aff/B_\aff)\to h_{T_\aff}(\Gr_G)$. We first compute the image of the two bases via this  map. 
\begin{lemma}Let $F=F_\fc$. For any $w\in W_\aff$ and $u=u_1u_2, u_1\in W^-_\aff, u_2\in W$, we have\[
Y_{w_0}\bullet \hat X_{u_1u_2}^*=\ep_{u_2}\fc_{w_0}\fc_{u_2}^{-1}\hat X_{u_1}^*, ~Y_{w_0}\bullet \hat Y_{w}^*=\sum_{v_1v_2\ge w, v_1\in W_\aff^-, v_2\in W}\ep_w\ep_{v_2}\fc_{v_1w_0}\fc_w^{-1}\hat X_{v_1}^*.
\]
In particular, $Y_{w_0}\bullet\hat Y_w^*, w\in W_\aff^-$ is a basis of $(\hatD_{W_\aff}^*)^W$ if and only if $\fc\in R^\times$.
\end{lemma}
\begin{proof}
For each $v\in W_\aff$, write $v=v_1v_2,v_1\in W_\aff^-, v_2\in W$. From $\hat X_wY_{w_0}=0, w\in W$, we have
\[
(Y_{w_0}\bullet \hat X_{u_1u_2}^*)(\hat X_{v_1v_2})=\hat X_{u_1u_2}^*(\hat X_{v_1v_2}Y_{w_0})=\de_{v_2,e}\hat X_{u_1u_2}^*(\hat X_{v_1}\sum_{w'\le w_0}\ep_{w'}\fc_{w_0}\fc_{w'}^{-1}\hat X_{w'})=\de_{v_2,e}\de_{v_1,u_1}\ep_{u_2}\fc_{w_0}\fc_{u_2}^{-1}. 
\]
This proves the first identity. For the second one, it is easy to see that $\hat Y_w^*=\sum_{v\ge w}\ep_w\fc_v\fc_w^{-1}\hat X_v^*.$
So 
\[
Y_{w_0}\bullet \hat Y_{w}^*=Y_{w_0}\bullet \sum_{v\ge w}\ep_w\fc_v\fc_w^{-1}\hat X_v^*=\sum_{v_1v_2\ge w, v_1\in W_\aff^-, v_2\in W}\ep_w\ep_{v_2}\fc_{v_1w_0}\fc_w^{-1}\hat X_{v_1}^*.
\]
This proves the second identity. 

The transition matrix between $\hat X_{v}^*, v\in W_\aff^-$ and $Y_{w_0}\bullet \hat Y_w^*, w\in W_\aff^-$ is upper triangular with diagonal entries $\ep_w\fc_{w_0}$, so  the last statement  follows. 
\end{proof}
\subsection{}Before computing the $\odot$-action, we need to prove some identities in $\hatD_{W_\aff}$. 
\begin{lemma}\label{lem:b}Let $F=F_\fc$. Writing $\eta_u=\sum_{v\le u}\hat b_{u,v}\hat X_v=\sum_{v\le u}\hat b^Y_{u,v}\hat Y_v$, then 
\begin{eqnarray*}
\hat b_{s_iu, v}&=&\left\{\begin{array}{cc}s_{i}(\hat b_{u,v}),& s_iv>v;\\
(1-\fc x_i)s_i(\hat b_{u,v})-x_is_i(\hat b_{u,s_iv}), & s_iv<v.\end{array} \right.\\
\hat b^Y_{s_iu, v}&=&\left\{\begin{array}{cc}(1-\fc x_i)s_i(\hat b^Y_{u,v}),& s_iv>v;\\
x_is_i(\hat b^Y_{u,s_iv})+s_i(\hat b^Y_{u,v}), & s_iv<v.\end{array} \right.
\end{eqnarray*}
\end{lemma}
\begin{proof}We prove the first one, and the second one follows similarly. Denote ${}^iW_\aff=\{v\in W_\aff|s_iv>v\}$. We have 
\begin{eqnarray*}
\eta_{s_iu}&=&\eta_i\eta_u=\eta_i\sum_{ v\in {}^iW_\aff}\hat b_{u,v}\hatX_v+\hat b_{u,s_iv}\hatX_{s_iv}=\sum_{v\in {}^iW_\aff}s_i(\hat b_{u,v})\eta_i\hatX_v+s_i(\hat b_{u,s_iv})\eta_i\hatX_{s_iv}\\
&=&\sum_{v\in {}^iW_\aff}s_i(\hat b_{u,v})(1-x_i\hatX_i)\hatX_v+s_i(\hat b_{u,s_iv})(1-x_i\hatX_i)\hatX_{s_iv}\\
&=&\sum_{v\in {}^iW_\aff}s_i(\hat b_{u,v})(\hatX_v-x_i\hatX_{s_iv})+s_i(\hat b_{u,s_iv})\hatX_{s_iv}-\fc x_is_i(\hat b_{u,s_iv})\hatX_{v}\\
&=&\sum_{v\in {}^iW_\aff}s_i(\hat b_{u,v})\hatX_v+(s_i(\hat b_{u,s_iv})(1-\fc x_i)-x_is_i(\hat b_{u,v}))\hatX_{s_iv}.
\end{eqnarray*}
The conclusion then follows.
\end{proof}
Note that if $v\in W_\aff^-$ and $s_iv<v$, then $s_iv\in W_\aff^-$. We have the following recursive formula, whose proof follows from the definition and Lemma \ref{lem:b}.
\begin{theorem}\label{thm:con}
For $F=F_\fc$, with $i\in I_\aff$, we have
\begin{eqnarray*}
\hatX_{-i}\odot \hatX_v^*&=&\left\{\begin{array}{cc}
0, &s_iv>v, \\
\fc\hatX_v^*+\hatX_{s_iv}^*, &s_iv<v,
\end{array}\right.\\
\hat Y_{-i}\odot \hat Y_v^*&=&\left\{\begin{array}{cc}
0, &s_iv>v, \\
\fc\hat Y_v^*+\hat Y_{s_iv}^*, &s_iv<v. 
\end{array}\right.
\end{eqnarray*}
Here \[\hat X_{-i}=\eta_{w_0}\hat X_i\eta_{w_0}=\frac{1}{x_{-i}}(1-\eta_i), ~\hat Y_{-i}=\eta_{w_0}\hat Y_i\eta_{w_0}=\frac{1}{x_i}+\frac{1}{x_{-i}}\eta_i.\]
Consequently, if $v\in W_\aff^-$, we have 
\[
\hat Y_{-i}\odot (Y_{w_0}\bullet Y_v^*)=\left\{\begin{array}{cc}
0, &s_iv>v, \\
\fc(Y_{w_0}\bullet \hat Y_v^*)+(Y_{w_0}\bullet \hat Y_{s_iv}^*), &s_iv<v. 
\end{array}\right.
\]
\end{theorem}
\begin{proof}We have 
\[
\hatX_{-i}\odot \hat X_v^*=(\frac{1}{x_{-i}}-\frac{1}{x_{-i}}\eta_i)\odot\prod_{u\ge v}\hat b_{u,v}\hat f_u=\prod_{u}\frac{\hat b_{u,v}}{x_{-i}}\hat f_u-\prod_u\frac{s_i(\hat b_{u,v})}{x_{-i}}\hat f_{s_iu}=\prod_u\frac{\hat b_{u,v}-s_i(\hat b_{s_iu,v})}{x_{-i}}\hat f_u.
\]
Plugging the formula in Lemma \ref{lem:b},  we obtain the formula. 

The formula for $\hat Y_{-i}\odot \hat Y_v^*$ follows similarly. From the commutativity of the two actions $\bullet$ and $\odot$, one obtains the last statement. 
\end{proof}

\section{FADA for the small torus}
We repeat the  construction of FADA for the  small torus, which is very similar as above. 

\subsection{}
Let $S$ be the formal group algebra associated to $T^*$, that is, it is (non-canonically) isomorphic  a power series  ring of rank $n$. When the formal group law $F=F_\fc$, we can again take the polynomial version, i.e., see Remark \ref{rm:poly}.  Let $\cQ=S[\frac{1}{x_\al},\al\in \Phi]$, $\cQ_{W_\aff}=\cQ\rtimes R[W_\aff]$, $\cQ_{Q^\vee}=\cQ\rtimes R[Q^\vee]$. For any $\al\in \Phi$, let $\ka_\al=\frac{1}{x_\al}+\frac{1}{x_{-\al}}$ and $\kappa_{\al_0}=\frac{1}{x_{-\theta}}+\frac{1}{x_\theta}$.  We have the projection
\[
\pr:\cQ_{W_\aff}\to \cQ_{Q^\vee}, ~\eta_{t_\la w}\mapsto \eta_{t_\la}, \quad w\in W.
\]
Define
\[
X_\al=\frac{1}{x_\al}(1-\eta_\al), \quad  X_{\al_0}=\frac{1}{x_{-\theta}}(1-\eta_{s_0}),\quad \al\in \Phi.
\]
For simplicity, denote $x_{\pm i}=x_{\pm \al_i}, X_i=X_{\al_i},  \eta_i=\eta_{s_i}$, $X_0=X_{\al_0}$. They satisfy relations similar as that of $\hatX_i$. One can define $X_{I_w}$  for any reduced sequence $I_w$ of $w$, which depends only on $w$ if $F=F_\fc$. 

\begin{remark}\label{rm:identifyK}Consider  K-theory, in which case  $F=F_\fc$ with $\fc=1$. Our  $-X_{-\al_i}$ is the $T_i$ in \cite{LSS10, LLMS18}. Our $1-X_{\al_i}$ coincides with the $D_i$ in \cite{K18}. For cohomology, $\fc=0$, $\kappa_\al=0$, and our $X_i$ is the $A_i$ in \cite[Proposition 2.11]{P97} and \cite{L06}. 
\end{remark}

\begin{lemma}\label{lem:przero}
We have $\pr(zX_i)=0$ if $z\in \cQ_{W_\aff}, i\in I$.
\end{lemma}
\begin{proof}Let $z=p\eta_w, p\in \cQ, w\in W_\aff$ , then 
\[
\pr(zX_i)=\pr(p\eta_wX_i)=\pr(\frac{p}{w(x_i)}(\eta_w-\eta_{ws_i}))=\frac{p}{w(x_i)}(\pr(\eta_w)-\pr(\eta_{ws_i}))=0. 
\]
\end{proof}

Define $\bfD_{W_\aff}$ to be the subalgebra of $\cQ_{W_\aff}$ generated by $S$ and $X_i, i\in I_\aff$, and $\bfD_{W_\aff/W}=\pr(\bfD_{W_\aff})$. Then $\bfD_{W_\aff}$ is a free left  $S$-module with basis $X_{I_w}, w\in W_\aff$.  Denote $\frakX_{I_w}=\pr(X_{I_w}), w\in W_\aff^-$. 
\begin{lemma}If $I_w,w\in W_\aff$ are $W$-compatible,  then the set $\{\frakX_{I_w}|w\in W_\aff^-\}$ is a basis of the left $S$-module $\bfD_{W_\aff/W}$.
\end{lemma}
\begin{proof}They follow easily from Lemma \ref{lem:przero}. See \cite[Lemma 11.3]{CZZ2}.
\end{proof}

The projection $\frakp:T_\aff^*\to T^*, \mu+k\de\mapsto \mu$ induces  projections $\hatS\to S$,  $\hatQ\to \cQ$ and $\hatQ_{W_\aff}\to \cQ_{W_\aff}$. Clearly  $\frakp(\hatX_{\al_i})=X_{\al_i}$ and $\frakp(\hatX_{I_w})=X_{I_w}$, so $\fp(\hatD_{W_\aff})=\bfD_{W_\aff}$. More explicitly,   we have
\begin{eqnarray*}
\hatX_{I_w}=\sum_{v\le w}\hat a_{I_w,v}\eta_v\in \hatQ_{W_\aff}, \quad X_{I_w}=\sum_{v\le w}a_{I_w,v}\eta_v\in \cQ_{W_\aff}, \quad \frakp(\hat a_{I_w,v})=a_{I_w,v}\in \cQ,\\
\eta_w=\sum_{v\le w}\hat b_{w,I_v}\hatX_{I_v}\in \hatQ_{W_\aff}, \quad \eta_w=\sum_{v\le w}b_{w,I_v}X_{I_v}\in \cQ_{W_\aff}, \quad \frakp(\hat b_{w,I_v})=b_{w,I_v}\in S. 
\end{eqnarray*}
Note that the embedding $i:\cQ\to \hatQ$  induces a section $\cQ_{W_\aff}\to \hatQ_{W_\aff}$ of $\frakp$. However, it does not map $\bfD_{W_\aff}$ to $\hatD_{W_\aff}$. For example, $X_0$ is mapped to $\frac{x_{-\theta+\de}}{x_{-\theta}}\hatX_0$ which does not belong to $ \hatD_{W_\aff}$. 

\subsection{}
As before, we can take the duals, which will give us $\cQ$-modules  $\cQ_{W_\aff}^*$, $\cQ_{Q^\vee}^*$, and $S$-modules  $\bfD_{W_\aff}^*$,  $\bfD_{W_\aff/W}^*$. The elements dual to 
\[\eta_w,X_{I_w}\in \bfD_{W_\aff}\subset \cQ_{W_\aff}, ~\eta_{t_\la}, \frakX_{I_w}\in \bfD_{W_\aff/W}\subset \cQ_{Q^\vee},\]
are denoted by 
\[f_w, X_{I_w}^* \in \bfD_{W_\aff}^*\subset \cQ_{W_\aff}^*, ~ f_{t_\la}, \frakX_{I_w}^* \in \bfD_{W_\aff/W}^*\subset \cQ_{Q^\vee}^*,\] correspondingly. Note that the notation $f_{t_\la}$ can be thought as in $\cQ_{W_\aff}^*$ and $\cQ_{Q^\vee}^*$, just like $\eta_{t_\la}$ can be thought as in $\cQ_{W_\aff}$ and $\cQ_{Q^\vee}$. Similar as Proposition \ref{prop:bigtorusGKM}, we have
\begin{equation}\label{eq:smalltorusmap}
\bfD_{W_\aff}^*=\{f\in \cQ_{W_\aff}^*|f(\bfD_{W_\aff})\subset S\}. 
\end{equation}
Moreover, by definition, the dual map $\pr^*:\cQ_{Q^\vee}^*\to \cQ_{W_\aff}^*$ satisfies
\[
\pr^*(f_{t_\la})=\sum_{w\in W}f_{t_\la w}.
\]

Following from the definition, we have
\[
\hatX_{I_w}^*=\prod_{v\ge w}\hat b_{v,I_w}\hat f_v\in \hatD_{W_\aff}^*, \quad X_{I_w}^*=\prod_{v\ge w} b_{v,I_w} f_v\in \bfD_{W_\aff}^*.
\]
Since  $\frakp(\hat b_{v,I_w})=b_{v,I_w}$, so the map $\frakq:\hatQ_{W_\aff}^*\to \cQ_{W_\aff}^*, \prod_{w}a_w\hat f_w\mapsto \prod_w\frakp(a_w)f_w$ induces a map $\frakq:\hatD_{W_\aff}^*\to \bfD_{W_\aff}^*$ such that $\frakq(\hatX^*_{I_w})=X_{I_w}^*$. Moreover, 
since \[\frakp^*(X_{I_w}^*)(\hatX_{I_v})=X_{I_w}^*(\frakp(\hatX_{I_v}))=X_{I_w}^*(X_{I_v})=\de_{w,v},\] so $\frakp^*(X_{I_w}^*)=\hatX_{I_w}^*.$ Note that neither $\frakq$ nor $\frakp^*$ are  isomorphisms, since the domains and targets are modules over different rings. 

Similar as Lemma \ref{lem:bigWinv}, we have
\begin{lemma}If $I_w, w\in W_\aff$ are $W$-compatible, then the set $X_{I_w}^*, w\in W_\aff^-$ form a basis of $(\cQ_{W_\aff}^*)^W$ and of  $(\bfD_{W_\aff}^*)^W$, respectively. 
\end{lemma}

\begin{lemma}\label{lem:frakX*}Assume that $\{I_w,w\in W_\aff\}$ is $W$-compatible. For any $w\in W, u\in W_\aff^-$, we have
\[
\mathfrak{X}_{I_u}^*=\prod_{\la\in Q^\vee}b_{t_\la w, I_{u}}f_{t_\la}\in \cQ_{Q^\vee}^*. 
\]
\end{lemma}
\begin{proof}
For any $\la\in Q^\vee$, write
\[\eta_{t_\la w}=\sum_{u\in W_\aff^-, v\in W}b_{t_\la w,I_{u}\cup I_v}X_{I_{u}\cup I_v}.\]
By Lemma \ref{lem:przero},  we have 
\[
\eta_{t_\la}=\pr(\eta_{t_\la w})=\sum_{u\in W_\aff^-, v\in W}b_{t_\la w, I_{u}\cup I_v}\pr(X_{I_{u}\cup I_v})=\sum_{u\in W_\aff^-}b_{t_{\la}w, I_{u}}\pr(X_{I_u})=\sum_{u\in W_\aff^-}b_{t_\la w, I_{u}} \frakX_{I_u}. 
\]
Therefore, 
\[
\frakX_{I_u}^*=\prod_{\la\in Q^\vee}b_{t_\la w, I_{u}}f_{t_\la}\in \cQ_{Q^\vee}^*. 
\]
\end{proof}
This lemma implies that we have $\pr^*(\frakX_{I_u}^*)=X_{I_u}^*$, $u\in W_\aff^-$.

\subsection{}
There is a $\bullet$-action of $\cQ_{W_\aff}$ on $\cQ_{W_\aff}^*$, defined similar as the big torus case.

\begin{lemma} \label{lem:smalltorusaction}The $\bullet$-action of $\cQ_{W_\aff}$ on $\cQ_{W_\aff}^*$restricts to an action of $\bfD_{W_\aff}$ on $\bfD_{W_\aff}^*$.
\end{lemma}
\begin{proof}Since $\bfD_{W_\aff}$ is a $S$-module with basis $X_{I_u}, u\in W_\aff$, so for any $w,v\in W_\aff, i\in I_\aff$, we have $X_{I_v}X_i=\sum_{u}c_{I_{v}\cup s_i,I_ u}X_{I_u}$ with $c_{I_v\cup s_i,I_u}\in S$. We have 
\[(X_i\bullet X_{I_w}^*)(X_{I_v})=X_{I_w}^*(X_{I_v}X_i)=c_{I_v\cup s_i, I_w}\in S.\]
By \eqref{eq:smalltorusmap},   $X_i\bullet X_{I_w}^*\in \bfD_{W_\aff}^*$. 
\end{proof}

\begin{lemma}We have 
\[
\bfD_{W_\aff}^*\subset \{f\in \cQ_{W_\aff}^*|f(\eta_w)\in S, \text{ and } f(\eta_w-\eta_{s_\al w})\in x_\al S, \forall \al\in \Phi, w\in W_\aff\}.
\]
\end{lemma}
One of the  main results of this paper is to study how different  the two sets are, that is, to derive the small torus GKM condition. 
\begin{proof}Since $\eta_w\in \bfD_{W_\aff}$, then it follows from \eqref{eq:smalltorusmap} that $f(\eta_w)\in S$. Let $i\in I$ and  $f=\prod_{w\in W_\aff}a_wf_w\in \bfD^*_{W_\aff}$ with $a_w=f(\eta_w)\in S.$ We have
\[
X_i\bullet f=\frac{1}{x_i}(1-\eta_i)\bullet \prod_wa_wf_w=\prod_w\frac{a_w}{w(x_i)}f_w-\prod_w\frac{a_w}{ws_i(x_i)}f_{ws_i}=\prod_w\frac{a_w-a_{ws_i}}{w(x_i)}=\prod_w\frac{a_w-a_{s_{w(\al_i)}w}}{x_{w(\al_i)}}f_w.
\]
By Lemma \ref{lem:smalltorusaction}, $X_i\bullet f\in \bfD_{W_\aff}^*$, so $f(\eta_w-\eta_{s_\be w})=\frac{a_w-a_{s_\be w}}{x_\be}\in S$ for any $\be\in \Phi.$
\end{proof}

\subsection{}
We can similarly define the $\odot$ action
\[
a\eta_w\odot bf_v=aw(b)f_{wv}, ~w,v\in W_\aff, a,b\in \cQ. 
\]
It is easy to see that the $\odot$ and the $\bullet$ actions commute with each other. 
\begin{lemma}For any $\hat z\in \hatQ_{W_\aff}, \hat f\in \hatQ_{W_\aff}^*$, we have 
\[
\fp(\hat z)\odot \fq(\hat f)=\fq(\hat z\odot \hat f).
\]
In particular, the $\odot$-action of $\cQ_{W_\aff}$ on $\cQ_{W_\aff}^*$ induces an action of $\bfD_{W_\aff}$ on $\bfD_{W_\aff}^*$.
\end{lemma}
\begin{proof}Write $\hat z=\hat a\eta_v, \hat f=\hat b\hat f_w, \hat a, \hat b\in \hatQ, w,v\in W_\aff$ and suppose $\fp(\hat a)=a, \fp(\hat b)=b$, then 
\[
\fp(\hat z)\odot \fq(\hat f)=a\eta_v\odot bf_w=av(b)f_{vw}=\fq(\hat a v(\hat b)\hat f_{vw})=\fq(\hat a\eta_v\odot \hat b\hat f_w)=\fq(\hat z\odot \hat f).
\]

For the second part, note that $\fp:\hatD_{W_\aff}\to \bfD_{W_\aff}$ and $\fq:\hatD_{W_\aff}^*\to \bfD_{W_\aff}^*$ are both surjective. Given $z\in \bfD_{W_\aff}$ and $f\in \bfD_{W_\aff}^*$, suppose $z=\fp(\hat z)$ and $f=\fq(\hat f)$ for some $\hat z\in \hatD_{W_\aff}$ and $\hat f\in \hatD_{W_\aff}^*$, then 
\[
z\odot f=\fp(\hat z)\odot \fq(\hat f)=\fq(\hat z\odot \hat f)\in \fq(\hatD_{W_\aff}^*)=\bfD_{W_\aff}^*.
\]
\end{proof}

\begin{remark}If $F=F_\fc$, then all results in \S\ref{sec:connective} holds for $X_w^*$ and the corresponding $Y_w^*$. 
\end{remark}

\section{The small-torus GKM condition}
In this section, we study the small-torus GKM condition on the equivariant oriented cohomology of the affine flag variety and of the affine Grassmannian. 
\subsection{}
For each $\al\in \Phi$, we define
\[
Z_{\al}=\frac{1}{x_{-\al}}(1-\eta_{t_{\al^\vee}})\in \cQ_{W_\aff}.
\]
\begin{lemma} For each $\al\in \Phi$, we have $Z_{\al}\in \bfD_{W_\aff}$.
\end{lemma}
\begin{proof} It suffices to   show that $Z_{\alpha}$ is contained in the subalgebra of $\bfD_{W_\aff}$ generated by $S$ and $X_\alpha$.  So we assume the root system is the affine root system of $\SL_2$ with simple roots $\al_1=\al, \al_0=-\al+\de$. Then $t_{\al^\vee}=s_0s_1$. We have $\eta_{s_1}=1-x_{\al}X_1, \eta_{s_0}=1-x_{-\al}X_0$, so  $\eta_{s_0s_1}=1-x_{-\al}X_0-x_{-\al}X_1+x_{-\al}^2X_0X_1$. Therefore,
\[
Z_\al=\frac{1}{x_{-\al}}(1-\eta_{s_0s_1})=X_0+X_1-x_{-\al}X_{0}X_1\in \bfD_{W_\aff}.
\] 
\end{proof}

\begin{example}Suppose the root system is $\hat A_1$  with two simple roots $\al_1=\al, \al_0=-\al+\de$. 
\begin{enumerate}
\item If $F=F_\fc$ with $\fc=0$,  then we have $Z_{\al}=X_{0}+X_{1}+\al X_{0}X_{1}$. 
\item If $F=F_\fc$ with $\fc=1$,  then we have $Z_\al=X_0+X_1+(e^\al-1)X_{0}X_1$. 
\end{enumerate}
\end{example}

Since $\bfD_{W_\aff}$ acts on $\bfD^*_{W_\aff}$, so we know that $Z_\alpha$ acts on $\bfD^*_{W_\aff}$. Note that 
\[
Z_\al^k=\frac{1}{x_\al^k}(1-\eta_{t_{\al^\vee}})^k. 
\]

\subsection{}
We are now ready to prove the first main result of this paper. 
\begin{theorem}\label{thm:GKM}
\begin{enumerate}
\item The subset $\bfD_{W_\aff}^*\subset \cQ_{W_\aff}^*$ consists of elements  satisfying the following small-torus GKM condition: 
\[f\left((1-\eta_{t_{\al^\vee}})^d\eta_w\right)\in x_\al^d S, \text{ and } f\left((1-\eta_{t_{\al^\vee}})^{d-1}(1-\eta_{s_\al})\eta_w\right)\in x_\al^d S, ~\forall \al\in \Phi, w\in W_\aff, d\ge 1.\]
\item The subset  $(\bfD_{W_\aff}^*)^W\subset (\cQ_{W_\aff}^*)^W$ consists of elements   satisfying the following small-torus Grassmannian  condition: 
\[
f\left((1-\eta_{t_{\al^\vee}})^d\eta_w\right)\in x_\al^d S, ~ \forall \al\in \Phi, w\in W_\aff,  d\ge 1. 
\]
\end{enumerate}
\end{theorem}
Our proof follows similarly as that of \cite[Theorem 4.3]{LSS10}. The key improvement is that we don't need to prove Propositions 4.4 and 4.5 of loc.it., since we can use the operators $Z_\alpha$. However, for the convenience of the readers, we include an appendix, which gives all coefficients of $b_{w,I_v}$ in the $\hat A_1$ case. They can be used to show that $X_{I_w}^*$ satisfy the small torus GKM condition. 
\begin{proof}
(1). We  prove that elements of $\bfD^*_{W_\aff}$ satisfy the small-torus GKM condition.
Let $f=\prod_wc_wf_w\in \bfD^*_{W_\aff}$, we have
\[
Z_\alpha\bullet \prod_wc_wf_w=\prod_w(\frac{c_w}{w(x_{-\al})}f_w-\frac{c_w}{wt_{-\al^\vee}(x_{-\al})}f_{wt_{-\al^\vee}})=\prod_{w}\frac{c_w-c_{t_{w(\al^\vee)}w}}{x_{-w(\al)}}f_w\in \bfD_{W_\aff}^*.
\]
Note that $\frac{x_\al}{x_{-\al}}$ is invertible in $S$. Therefore, denoting $w(\al)=\be$, by \eqref{eq:smalltorusmap}, we have $f((1-\eta_{t_{\be^\vee}})\eta_w)\in x_\be S$ for any $\be\in \Phi$. 

Moreover, denote $d_w=\frac{c_w-c_{wt_{\al^\vee}}}{x_{-w(\al)}}$, then $d_{wt_{\al^\vee}}=\frac{c_{wt_{\al^\vee}}-c_{wt_{\al^\vee}t_{\al^\vee}}}{x_{-wt_{\al^\vee}(\al)}}=\frac{c_{wt_{\al^\vee}}-c_{wt_{2\al^\vee}}}{x_{-w(\al)}}$. Therefore, 
We have 
\begin{eqnarray*}Z_\al^2\bullet f&=&Z_\al\bullet Z_\al\bullet \prod_w c_w f_w
=\prod_w (\frac{d_w-d_{wt_{\al^\vee}}}{w(x_{-\al})})f_w=\prod_w \frac{c_w-2c_{wt_{\al^\vee} }+c_{wt_{2\al^\vee}}}{w(x_{-\al})^2}f_w\\
&=&\prod_w\frac{c_w-2c_{t_{w(\al^\vee)}w}+c_{t_{2w(\al^\vee)}w}}{x_{-w(\al)}^2}f_w=\prod_w\frac{1}{x^2_{-w(\al)}}f((1-\eta_{t_{w(\al^\vee)}})^2\eta_w)f_w.
\end{eqnarray*}
Denoting $w(\al)=\be$, we see that $f((1-\eta_{t_{\be^\vee}})^2\eta_w) \in x_\be^2 S$. Inductively, we see that  $f((1-\eta_{t_{\al^\vee}})^d\eta_w)\in x_\al^d S$ for all $d\ge 1$.

Similarly, if one applies  $Z_\al^{d-1}X_\al\in \bfD_{W_\aff}$ on $f$, which gives $Z_\al^{d-1}X_\al\bullet f\in \bfD_{W_\aff}^*$, one will see that $f$ satisfies the second condition. 

For the rest of the proof and for that of (2), it is identical to that of \cite[Theorem 4.3]{LSS10}, so it is skipped. 
\end{proof}
\begin{corollary}\label{cor:GrassGKM}The subset $\bfD_{W_\aff/W}^*\subset \cQ_{Q^\vee}^*$ consists of elements satisfying the following small torus Grassmannian condition:
\[
f((1-\eta_{t_\al^\vee})^d\eta_{t_\la})\in x_\al^d S, ~\forall \al\in \Phi, d\ge 1, \la\in Q^\vee. 
\]
\end{corollary}
\begin{proof}This follows from the identity $\pr^*(f_{t_\la})=\sum_{v\in W}f_{t_\la v}$. 
\end{proof}

\section{The Peterson subalgebra}

In this section, we embed $\bfD_{W_\aff/W}$ into $\bfD_{W_\aff}$ and show that it coincides with  the centralizer of $S$ in $\bfD_{W_\aff}$. This is called the Peterson subalgebra, which gives the algebraic model for the equivariant oriented `homology' of  the affine Grassmannian.

\subsection{}
We have a canonical ring embedding (and also a $\cQ$-module embedding)
\begin{eqnarray*}
k:\cQ_{Q^\vee}\to \cQ_{W_\aff}, &\quad p\eta_{t_\la}\mapsto p\eta_{t_\la}, ~
\end{eqnarray*}
such that $\pr \circ k=\id_{\cQ_{W_\aff}}$. It is easy to see that the dual map $k^*:\cQ_{W_\aff}^*\to \cQ_{Q^\vee}^*$ satisfies
\begin{equation}\label{eq:k*}
 \quad k^*(f_{t_\la u})=\de_{u,e}f_{t_\la}, \quad u\in W. 
\end{equation}
 For K-theory, our map  $k$ is the map $k:K_T(\Gr_G)\to \bbK$ in \cite[\S5.2]{LSS10}, and $k^*$ is the wrong-way map $\varpi$ of \cite[\$4.4]{LSS10}.

The following lemma generalizes \cite{P97}, \cite[Theorem 4.4]{L06}  for the cohomology case, and \cite[Lemma 4.6]{LSS10} for the K-theory case. 
\begin{lemma}\label{lem:k*} The map $k^*$ induces a map $k^*:\bfD_{W_\aff}^*\to \bfD_{W_\aff/W}^*$. Consequently, the map $k$ induces a map $k:\bfD_{W_\aff/W}\to \bfD_{W_\aff}$. 
\end{lemma}
\begin{proof}
Given $f\in \bfD_{W_\aff}^*$, then $f$ satisfies the small-torus GKM condition Theorem \ref{thm:GKM}, that is, 
\[
f((1-\eta_{t_{\al^\vee}})^d\eta_{t_\la u})\in x_\al^d S, \quad \forall u\in W, \la\in Q^\vee, \al\in \Phi, d\ge 1.
\]
Therefore, 
\[
k^*(f)((1-\eta_{t_{\al^\vee}})^d\eta_{t_\la })=f\left(k((1-\eta_{t_{\al^\vee}})^d\eta_{t_\la })\right)=f((1-\eta_{t_\al^\vee})^d\eta_{t_\la})\in x_\al^d S. 
\]
 Therefore, by Corollary \ref{cor:GrassGKM},   $k^*(f)\in \bfD_{W_\aff/W}^*$. 
\end{proof}

\begin{remark}It would be interesting to find a direct proof of the fact that $k$ maps $\bfD_{W_\aff/W}$ to $\bfD_{W_\aff}$. One possible choice is to find the small torus residue condition of $\bfD_{W_\aff} $ similar to the residue condition of \cite{GKV97} (see  \cite{ZZ17}).
\end{remark}

\begin{example}Note that this result is not true for the big torus case, that is, $ k(\hatD_{W_\aff/W})$  is not contained in $\hatD_{W_\aff}$. For example, in the  $\hat A_1$ case, we have
\[
\pr(X_0)=\pr(\frac{1}{x_{-\al+\de}}(1-\eta_{t_{\al^\vee}s_1}))=\frac{1}{x_{-\al+\delta}}(1-\eta_{t_{\al^\vee}})\in \hatD_{W_\aff/W},
\] and 
\[
k(\pr(X_0))=\frac{1}{x_{-\al+\de}}(1-\eta_{t_\al})=\frac{1}{x_{-\al+\de}}(1-\eta_{s_0s_1})\not\in \hatD_{W_\aff}. 
\]
\end{example}
\begin{lemma}\label{lem:frakX*k*} If $I_w,w\in W$ is $W$-compatible, then $k^*(X_{I_u}^*)=\frakX_{I_u}^*$ for any $u\in W_\aff^-$. 
\end{lemma}
\begin{proof}By \eqref{eq:k*} and Lemma \ref{lem:frakX*}, we have
\[
k^*(X_{I_u}^*)=k^*(\prod_{\la\in Q^\vee, w\in W}b_{t_\la w,I_u}f_{t_\la w})=\prod_{\la\in Q^\vee}b_{t_\la,I_u }f_{t_\la}=\frakX_{I_u}^*.
\]
\end{proof}

\subsection{}
Let $C_{\bfD_{W_\aff}}(S)$ be the centralizer of $S$ in $\bfD_{W_\aff}$.
Our second main result is the following, which generalizes \cite[Lemma 5.2]{LSS10} in the K-theory case and  \cite[\S9.3]{P97} in the cohomology case (proved in \cite[Theorem 6.2]{LS10}). 
\begin{theorem}\label{thm:Peterson}
We have $C_{\bfD_{W_\aff}}(S)=k(\cQ_{Q^\vee})\cap \bfD_{W_\aff}= k(\bfD_{W_\aff/W})$. 
\end{theorem}
\begin{proof}
We look at the first identity. Since ${t_\la}(p)=p$ for any $ p\in S$, so it is clear that $\cQ_{Q^\vee}\cap \bfD_{W_\aff}\subset C_{\bfD_{W_\aff}}(S)$. Conversely, let $z=\sum_{w\in W_\aff}c_w\eta_w\in C_{\bfD_{W_\aff}}(S)$, then for any $\mu\in T^*$, we have 
\[
0=x_\mu z-zx_\mu=\sum_{w\in W_\aff}c_w(x_\mu-x_{w(\mu)})\eta_w. 
\]
Therefore, for any $c_w\neq 0$, we have $\mu=w(\mu)$ for all $\mu\in T^*$. we can take $\mu$ to be $W$-regular, which shows that $c_w\neq 0$ only when $w=t_\la$ for some $\la\in Q^\vee$. So $z\in k(\cQ_{Q^\vee})$. The first identity is proved. 

We now look at the second identity. It follows from Lemma \ref{lem:k*} that $k(\bfD_{W_\aff/W})\subset k(\cQ_{Q^\vee})\cap  \bfD_{W_\aff}$. For the other inclusion, note that  $\eta_{t_\la}\in \bfD_{W_\aff}$ is a $\cQ$-basis of $k(\cQ_{Q^\vee})$. Given any $z=\sum_{\la\in Q^\vee}p_\la \eta_{t_\la}\in k(\cQ_{Q^\vee})\cap \bfD_{W_\aff}, p_\la\in \cQ$, then $\pr(z)\in \pr(\bfD_{W_\aff})=\bfD_{W_\aff/W}$, and 
\[k\circ \pr(z)=k\circ \pr(\sum_\la p_\la \eta_{t_\la})=k(\sum_{\la }p_\la \eta_{t_\la})=\sum_\la  p_\la \eta_{t_\la}=z.\]
Therefore, $k(\cQ_{Q^\vee})\cap \bfD_{W_\aff}\subset k(\bfD_{W_\aff/W})$. The second identity is proved. 
\end{proof}

\begin{definition}We define the Peterson subalgebra to be  $\bfD_{Q^\vee}=k(\bfD_{W_\aff/W})$. 
\end{definition}
Let $I_w, w\in W_\aff$ be $W$-compatible. Since $\bfD_{W_\aff/W}$ is a free $S$-module with basis $\frakX_{I_w}, w\in W_\aff^-$, so $k(\frakX_{{I_w}})$ form a basis of $\bfD_{Q^\vee}$. This is the algebraic model for the oriented homology of the affine Grassmannian $\Gr_G$. 
The following result generalizes \cite[Theorem 5.3]{LSS10} in K-theory. 
\begin{theorem}The ring $\bfD_{Q^\vee}$ is a Hopf algebra, and the embedding $\bfD_{Q^\vee}\to \cQ_{Q^\vee}$ is an Hopf-algebra homomorphism.
\end{theorem}
\begin{proof}The coproduct structure on $\cQ_{W_\aff}$ is defined as $\triangle: \cQ_{W_\aff}\to \cQ_{W_\aff}\otimes_{\cQ}\cQ_{W_\aff}, \eta_w\mapsto \eta_w\otimes \eta_w.$ It is easy to see that this induces a coproduct structure on $\cQ_{Q^\vee}$, and by  \cite{CZZ1}, it induces a coproduct structure on $\bfD_{W_\aff}$. Therefore, it induces a coproduct structure on  $\bfD_{Q^\vee}$. The product structure is induced by that of $\cQ_{Q^\vee}$, The antipode  is $\fraks:\cQ_{Q^\vee}\to \cQ_{Q^\vee}, \eta_{t_\la}\mapsto \eta_{t_{-\la}}$. It is then routine to check that $\bfD_{Q^\vee}$ is a Hopf algebra and the embedding  to $\cQ_{Q^\vee}$ is an embedding of Hopf algebras.
\end{proof}
\begin{remark}For K-theory, we know the Hecke algebra is contained $\bfD_{W_\aff}$. It is proved by Berenstein-Kazhdan \cite{BK19} that certain localization of the Hecke algbra is a Hopf algebra. It is not difficult to see that it is compatible with the Hopf algebra structure of $\bfD_{Q^\vee}$. 
\end{remark}

\subsection{}
The following theorem generalizes \cite[Theorem 5.4]{LSS10} in the K-theory case and \cite[Theorem 6.2]{LS10} in the cohomology case.  
\begin{theorem}\label{thm:KX} Assume $I_w,w\in W_\aff$ is $W$-compatible. If $u\in W_\aff^-$, then we have 
\[
k(\frakX_{I_u})=X_{I_u}+\sum_{v\in W_\aff\backslash W_\aff^-}c_{I_u,I_v}X_{I_v}, ~c_{I_u, I_v}\in S. 
\]
\end{theorem}
\begin{proof}
If $w\in W_\aff^-$, by Lemma \ref{lem:frakX*k*}, we have
\[
X_{I_w}^*(k(\frakX_{I_u}))=k^*(X_{I_w}^*)((\frakX_{I_u}))=\frakX_{I_w}^*(\frakX_{I_u})=\de_{w,u},
\] Therefore, 
\[
k(\frakX_{I_u})=\sum_{v\in W_\aff}X_{I_v}^*(k(\frakX_{I_u}))X_{I_v}=X_{I_u}+\sum_{v\in W_\aff\backslash W_\aff^-}c_{I_u,I_v}X_{I_v}. 
\]
\end{proof}
\begin{example}Consider the $\hat A_1$ case, then there are two simple roots $\al_1=\al, \al_0=-\al+\de$. By direct computation, we have
\begin{enumerate}
\item $k(\frakX_0)=X_0+X_1-x_{-\al}X_{01}$. 
\item $k(\frakX_{10})=X_{10}-\frac{x_{-\al}}{x_\al} X_{01}$.
\item $k(\frakX_{010})=X_{010}+X_{101}-x_{-\al}X_{1010}$. 
\end{enumerate}
\end{example}
\begin{corollary}Assume $I_w, w\in W_\aff$ is $W$-compatible.  Let $u,v\in W_\aff^-$. Write 
\[X_{I_{u}}X_{I_v}=\sum_{w\in W_\aff}d^{I_w}_{I_u,I_v}X_{I_w}, ~ \frakX_{I_u}\frakX_{I_v}=\sum_{w\in W_\aff^-}\mathfrak{d}^{I_w}_{I_u,I_v}\frakX_{I_w},\]
 then 
\[\frakd^{I_{w_3}}_{I_u,I_v}=\sum_{w_2\in W_\aff}c_{I_u,I_{w_2}}d^{I_{w_3}}_{I_{w_2}, I_{v}}.  \]
\end{corollary}
\begin{proof}We have
\begin{eqnarray*}
k(\sum_{w\in W_\aff^-}\mathfrak{d}^{I_w}_{I_u,I_v}\frakX_{I_w})&=&k(\frakX_{I_u}\frakX_{I_v})=k(\frakX_{I_u})k(\frakX_{I_v})=k(\frakX_{I_u})\sum_{w_1\in W_\aff}c_{I_v,I_{w_1}}X_{I_{w_1}}\\
&=&\sum_{w_1\in W_\aff}c_{I_v,I_{w_1}}k(\frakX_{I_u})X_{I_{w_1}}=\sum_{w_1,w_2\in W_\aff}c_{I_v,I_{w_1}}c_{I_u,I_{w_2}}X_{I_{w_2}}X_{I_{w_1}}.
\end{eqnarray*}
Let $w_3\in W_\aff^-$. By \cite[Theorem 8.2]{CZZ2}, we know that $X_{I_{w_3}}^*(X_{I_{w_2}}X_{I_{w_1}})=0$ unless $w_1\in W_\aff^-$, in which case $c_{I_v,I_{w_1}}=\de^{\Kr}_{v,w_1}$ by Theorem \ref{thm:KX}. Therefore, applying $X_{I_{w_3}}^*, w_3\in W_\aff^-$, and using Lemma \ref{lem:frakX*k*}, we get 
\begin{eqnarray*}
\frakd^{I_{w_3}}_{I_u,I_v}&=&\frakX_{I_{w_3}}^*(\sum_{w\in W_\aff^-}\frakd^{I_w}_{I_u,I_v}\frakX_{I_w})=k^*(X_{I_{w_3}}^*)(\sum_{w\in W_\aff^-}\frakd^{I_w}_{I_u,I_v}\frakX_{I_w})\\
&=&X_{I_{w_3}}^*(k(\sum_{w\in W_\aff^-}\frakd^{I_w}_{I_u, I_v}\frakX_{I_w}))=X^*_{I_{w_3}}(\sum_{w_1,w_2\in W_\aff}c_{I_v,I_{w_1}}c_{I_u,I_{w_2}}X_{I_{w_2}}X_{I_{w_1}})\\
&=&\sum_{w_2\in W_\aff}c_{I_u, I_{w_2}}X_{I_{w_3}}^*(X_{I_{w_2}}X_{I_v})=\sum_{w_2\in W_\aff}c_{I_u,I_{w_2}}d^{I_{w_3}}_{I_{w_2}, I_{v}}. 
\end{eqnarray*}
\end{proof}

\section{Appendix: Restriction formula  in the $\hat A_1$ case}

In this Appendix, we perform some computation in the $\hat A_1$ case. 

\subsection{}
In this case, there are two  simple roots, $\al_1=\al, \al_0=-\al+\de$, and any $w\in W_\aff$ has a unique reduced decomposition, so $X_{I_w}, Y_{I_w}$ can be denoted by $X_w, Y_w$, respectively. Moreover, $X_i^2=\kappa_\al X_i. $
We use the notation as in \cite[\S4.3]{LSS10}. Let 
\[\sigma_0=e, ~\sigma_{2i}=(s_1s_0)^i=t_{-i\al^\vee},~ \sigma_{-2i}=(s_0s_1)^i=t_{i\al^\vee},~ \sigma_{2i+1}=s_0\sigma_{2i}, ~\sigma_{-(2i+1)}=s_1\sigma_{-2i}, ~i\ge 1, \]
 and $W^-_\aff=\{\sigma_i|i\ge 0\}$. 
Denote $\mu=-\frac{x_{-1}}{x_1}$. So if $F=F_\fc$ with $\fc=0$, then $\mu=1$, and if $F=F_\fc$ with $\fc=1$, then $\mu=e^\al$ if one identifies $x_\al$ with $1-e^{-\al}$. 

Let $S_{\le a}$ be the sum $h_0+h_1+\cdots +h_a$ of homogeneous symmetric functions. Denote $S^i_{\le a}$ to be $S_{\le a}(x,x,\cdots, x)$ where there are $i$ copies of $x$. For instance, $S^3_{\le 3}(x)=1+3x+6x^2+10x^3$. We have the following identities:
\[
S^i_{\le a}(x)=xS^i_{\le a-1}(x)+S^{i-1}_{\le a}(x), \quad S^i_{\le a}(x)=\sum_{j=0}^ax^j\begin{pmatrix}j+i-1\\i-1\end{pmatrix}.
\]
Then the following identities can be verified by direct computation for lower $k$ and then continued with  induction: 
\begin{eqnarray*}
\eta_{\sigma_{2k}}&=&1+x_1^{2k}X_{\sigma_{2k}}+\sum_{1\le j\le k-1}x_1^{2j}(S^{2j}_{\le k-j}(\mu^{-1})X_{\sigma_{2j}}+S^{2j}_{\le k-j-1}(\mu^{-1})X_{\sigma_{-2j}})\\
&&-\sum_{1\le i\le k}x_1^{2i-1}S^{2i-1}_{\le k-i}(\mu^{-1})(X_{\sigma_{2i-1}}+X_{\sigma_{-2i+1}}), \\
\eta_{\sigma_{-2k}}&=&1+x_{-1}^{2k}X_{\sigma_{-2k}}+\sum_{1\le j\le k-1}x_{-1}^{2j}(S^{2j}_{\le k-j-1}(\mu)X_{\sigma_{2j}}+S^{2j}_{\le k-j}(\mu)X_{\sigma_{-2j}})\\
&&-\sum_{1\le i\le k}x_{-1}^{2i-1}S^{2i-1}_{\le k-i}(\mu)(X_{\sigma_{2i-1}}+X_{\sigma_{-2i+1}}), \\
\eta_{\sigma_{-2k-1}}&=&1-x_{1}^{2k+1}X_{\sigma_{-2k-1}}+\sum_{1\le j\le k}x_{1}^{2j}S^{2j}_{\le k-j}(\mu^{-1})(X_{\sigma_{2j}}+X_{\sigma_{-2j}})\\
&&-\sum_{1\le i\le k}x_{1}^{2i-1}\left(S^{2i-1}_{\le k-i}(\mu^{-1})X_{\sigma_{2i-1}}+S^{2i-1}_{\le k-i+1}(\mu^{-1})X_{\sigma_{-2i+1}}\right), \\
\eta_{\sigma_{2k+1}}&=&1-x_{-1}^{2k+1}X_{\sigma_{2k+1}}+\sum_{1\le j\le k}x_{-1}^{2j}S^{2j}_{\le k-j}(\mu)(X_{\sigma_{2j}}+X_{\sigma_{-2j}})\\
&&-\sum_{1\le i\le k}x_{-1}^{2i-1}\left(S^{2i-1}_{\le k-i+1}(\mu)X_{\sigma_{2i-1}}+S^{2i-1}_{\le k-i}(\mu)X_{\sigma_{-2i+1}}\right). 
\end{eqnarray*}
For $F=F_\fc$ with $\fc=1$, that is, in the K-theory case, these identities specializes to the corresponding ones in \cite[(4.5), (4.6)]{LSS10} after identifying our $-X_{-\al_i}$ with $T_i$ in \cite{LSS10} (see Remark \ref{rm:identifyK}). 
By using these identities, following the same idea as in \cite[\S4.3]{LSS10}, one can prove that $X_{I_w}^*$ satisfy the small torus GKM conditions in Theorem \ref{thm:GKM}. 

\subsubsection*{Acknowledge} The author would like to thank Cristian Lenart, Changzheng Li and  Gufang Zhao for helpful discussions.

\newcommand{\arxiv}[1]
{\texttt{\href{http://arxiv.org/abs/#1}{arXiv:#1}}}
\newcommand{\doi}[1]
{\texttt{\href{http://dx.doi.org/#1}{doi:#1}}}
\renewcommand{\MR}[1]
{\href{http://www.ams.org/mathscinet-getitem?mr=#1}{MR#1}}

\end{document}